\documentclass[12pt]{amsart}
\usepackage{amsmath,amsthm,amsopn,amssymb,anyfontsize,color,mathrsfs,times,newtxtext,newtxmath,upref}
\usepackage[mathscr]{eucal}
\usepackage{enumitem}

\newtheorem{theorem}{Theorem}[section]
\newtheorem{proposition}[theorem]{Proposition}
\newtheorem{corollary}[theorem]{Corollary}
\newtheorem*{corollary*}{Corollary}
\newtheorem{lemma}[theorem]{Lemma}
\newtheorem*{lemma*}{Lemma}
\numberwithin{equation}{section}
\theoremstyle{definition}
\newtheorem{definition}[theorem]{Definition}

\newtheorem{example}[theorem]{Example}
\newtheorem{remark}[theorem]{Remark}

\newcommand{\HH}{\mathcal{H}}

\newcommand{\M}{\mathcal{M}}
\newcommand{\N}{\mathcal{N}}
\newcommand{\PI}[2]{\left\langle \,#1 , #2\, \right\rangle}
\newcommand{\mc}[1]{\mathcal{#1}}

\DeclareMathOperator{\tr}{tr}

\begin{document}

\title{Total least squares problems on infinite dimensional spaces}

\author[Contino]{Maximiliano Contino}
\address[Maximiliano Contino]{Instituto Argentino de Matem\'atica ``Alberto P. Calder\'on'' \\ CONICET \\
	Saavedra 15, Piso 3 \\
	(1083) Buenos Aires, Argentina \& 
	Facultad de Ingenier\'{\i}a, Universidad de Buenos Aires \\
	Paseo Col\'on 850 \\
	(1063) Buenos Aires, Argentina.}
\email{mcontino@fi.uba.ar}

\author[Fongi]{Guillermina Fongi}
\address[Guillermina Fongi]{Centro Franco Argentino de Ciencias de la Informaci\'on y de Sistemas \\ CONICET \\ Ocampo y Esmeralda \\ (2000)  Rosario, Argentina.}
\email{gfongi@conicet.gov.ar}

\author[Maestripieri]{Alejandra Maestripieri}
\address[Alejandra Maestripieri]{Instituto Argentino de Matem\'atica ``Alberto P. Calder\'on'' \\ CONICET\\
	Saavedra 15, Piso 3\\
	(1083) Buenos Aires, Argentina \\
\& Facultad de Ingenier\'{\i}a, Universidad de Buenos Aires\\
Paseo Col\'on 850 \\
(1063) Buenos Aires, Argentina.}
\email{amaestri@fi.uba.ar}

\author[Muro]{Santiago Muro}
\address[Santiago Muro]{Centro Franco Argentino de Ciencias de la Informaci\'on y de Sistemas \\ CONICET \\ Ocampo y Esmeralda \\ (2000)  Rosario, Argentina.}
\email{muro@cifasis-conicet.gov.ar}


\begin{abstract} 
In this work we study weighted total least squares problems on infinite dimensional spaces.
We show that in most cases this problem does not admit a solution (except in the trivial case) and then, we consider a regularization on the problem. We present necessary conditions for the regularized problem to have a solution. We also show that, by restricting the regularized minimization problem to special subsets, the existence of a solution may be assured.
\end{abstract}

\maketitle

\section{Introduction}
In the classic least squares problem  \cite{nashed1987inner,rao1973theory}, to solve the inverse problem associated to 
a linear system $Ax=b$, where $A$ is a linear operator on a Hilbert space, the operator $A$ is assumed to be known exactly and only $b$ contains noise. However, this assumption may be unrealistic:  sampling errors, human errors,  modeling errors and instrument errors may imply inaccuracies of the operator $A.$ 

In the finite dimensional case, to obtain approximate solutions of a linear system $Ax=b$, with $A \in \mathbb{R}^{m\times n}$ and $b \in \mathbb{R}^m$,  in which it is assumed that both the elements of $A$ and $b$ are known up to some noise, a widely used  approach is to solve  the so-called {\it total least squares} (TLS) problem. The usual formulation consists in finding (if there exists) a solution to the minimization problem
$$
\underset{X \in \mathbb{R}^{m \times n}, \ x\in \mathbb{R}^{n} \ }{\min}  \Vert (A|b) - (X|y) \Vert_F^2, \qquad \textrm{subject to }Xx=y,
$$ 
where  $\Vert \cdot \Vert_F$ denotes the Frobenius norm and $(A|b)$ denotes the augmented matrix.

There are several examples of this minimization problems in signal processing, automatic control, biology, physics, and statistics (see \cite{van1991total} and its references). 
The TLS problem was studied for example in \cite{ golub1999tikhonov,golub1979total,van1991total},  
where an explicit solution, expressed by the singular value decomposition of the augmented matrix $(A|b)$, is given.

In many occasions (for example in the case of integral equations), the problem is originally set in infinite dimensional Hilbert spaces and the classic approach of \cite{golub1979total}, where singular value decomposition is used, is not available. 
The main goal of this paper is to study  total least squares and related inverse problems on infinite dimensional spaces. 
In particular, we discuss extensively the existence of solution on infinite dimensional spaces and we show that it is a delicate matter. 


Let us fix some notations: $\HH, \mc{F}$ are complex or real Hilbert spaces, $L(\HH,\mc{F})$ is the set of bounded linear operators from $\HH$ to $\mc{F}$ and $L(\HH) := L(\HH, \HH).$

Given $A \in L(\HH,\mc{F})$ and  $b \in \mc{F},$ in order to  formulate the TLS problem in infinite dimensional Hilbert spaces, we introduce a positive (semidefinite) weight $W \in L(\mc{F})$ such that $W^{1/2} \in S_2$ (the Hilbert-Schmidt class).
The problem is to find if there exists
\begin{equation} \label{eq WCMT}
\underset{X \in L(\HH, \mc{F}), \ x\in \HH}{\min} \Vert Xx-b \Vert_W^2 +\Vert  A - X\Vert_{2,W}^2,
\end{equation}
where $\Vert X \Vert_{2,W}=\Vert W^{1/2} X \Vert_2, \mbox{ for every } X \in L(\HH, \mc{F})$ and $\Vert z \Vert_W=\Vert W^{1/2} z \Vert, \mbox{ for every } z \in \mc F,$ the seminorms associated to $W.$ 
We will refer to problem \eqref{eq WCMT} as the \emph{weighted total least squares} (WTLS) problem.

One typical application of the total least squares problem is to ill-conditioned problems arising from integral equations. While these kind of application is of an infinite dimensional nature, the usual methods to solve the problem apply only to finite dimensional systems. 
Then, for such a problem, the total least squares problem is usually solved for a discretization of the original equation and we can not assure that this discretizations converge to the TLS associated to the original infinite dimensional problem. Moreover, it is not even clear whether this infinite dimensional TLS problem admits a solution. 
In fact, we will show that in most cases the TLS problem on infinite dimensional spaces does not admit a solution  unless in the trivial case (see Section \ref{section TLS}). 
We also exhibit a simple example where any finite dimensional approximation has solution but the original infinite dimensional problem does not (see  Example \ref{ejemplo intro}).

It is thus natural to consider a regularization on the TLS problem (even in the cases where the solution of the TLS problem exists  regularizations are often considered in  order to obtain a less contaminated solution, see for example \cite{beck2006solution,bjorck1996numerical,golub1999tikhonov, Tikhonov}). We study the extension of the so-called  \emph{Tikhonov regularization} to infinite dimensional spaces.
Given $A, T \in L(\HH,\mc{F}),$ $b \in \mc{F}$ and a positive (semidefinite) weight $W \in L(\mc{F})$ such that $W^{1/2} \in S_2,$ the problem is to find if there exists
\begin{equation} \label{eq WCMTR}
\underset{X \in L(\HH,\mc{F}), \ x\in \mc{H} \ }{\min} (\Vert T x \Vert^2+\Vert Xx-b \Vert_W^2 +\Vert  A - X\Vert_{2,W}^2).
\end{equation}
We will refer to problem \eqref{eq WCMTR} as the \emph{regularized weighted total least squares} (RWTLS) problem.

Recently, this problem has been studied on infinite dimensional Hilbert spaces, 
see \cite{bleyer2013double,bleyer2015alternating}. There, the authors stated general existence results under a hypothesis of weak to norm continuity of a certain bilinear application, which allows them to prove that the objective function is weakly lower semicontinuous. 
However, as we will see in Section 5, this type of continuity does not hold in many reasonable cases. In fact, in general this application is not even weak to weak continuous. We seek for conditions under which the continuity of the bilinear mapping may be assured (and hence the existence of solution of the RWTLS problem).
  
The paper is organized as follows. In Section \ref{Preliminaries}, we fix some notation and collect certain properties of the Hilbert-Schmidt class operators that will be used along the paper. In Section \ref{section TLS}, we prove that in most cases the non-regularized 
TLS problem does not have a solution on infinite dimensional  spaces. Therefore, we focus our attention on the regularized problem \eqref{eq WCMTR} on infinite dimensional Hilbert spaces in Section \ref{RWTLS}. We present necessary conditions for a pair $(A_0,x_0),$ to be a solution of the RWTLS problem. We observe that if the RWTLS problem has a solution $(A_0,x_0),$ then $x_0$ is a solution of the \emph{classical smoothing problem} \cite{atteia1965generalisation,[ChaLenMil96],[ChaLenMil00],corach2016optimal,corach2002oblique}
and the results obtained in \cite{contino2019global} can be applied for giving necessary conditions for the existence of solution of problem \eqref{eq WCMTR}.
In Section \ref{section case-I}, we study the case where the regularization is given by a multiple of the identity  operator, $\rho I$. We apply there the Dinkelbach method to observe that there exists a solution of the RWTLS problem, provided $\rho\ge t^*$, where $t^*$ is the infimum value of the RWTLS problem.

We show that $t^*$ can be obtained via a semidefinite programming problem.
In Section \ref{lower-semi}, we give several examples that show that, on infinite dimensional spaces, the objective function we need to minimize to solve the RWTLS problem is not generally weakly lower semicontinuous. 
We also show that restricting the minimization problem to special subsets, the semicontinuity and hence the existence of solution may be assured. 

\section{Preliminaries} \label{Preliminaries}
Throughout $E, E_0,E_1,E_2$ denote complex or real Banach spaces,
$L(E_0, E_1)$ is the set of bounded linear operators from $E_0$ to $E_1,$  $L(E):=L(E,E).$ Denote $E^*$ the dual space of $E$.
For any $A \in L(E_0,E_1),$ its range and nullspace are denoted by $R(A)$ and $N(A)$, respectively.

We recall the concept of ideal of operators (see for example \cite{defant1992tensor,pietsch1978operator}).
  We say that a class $\mathcal I$ of bounded linear operators and a norm $\|\cdot\|_{\mathcal I}$ form a normed ideal if for each set $\mathcal I(E_0,E_1) := \mathcal I \cap L(E_0,E_1)$ one has that $(\mathcal I(E_0,E_1),\|\cdot\|_{\mathcal I})$ is a normed space containing all finite rank operators such that 
\begin{enumerate}
    \item $TXS\in \mathcal I(E,E_2)$ for every $T\in L(E_1,E_2)$, $S\in L(E,E_0)$ and $X\in \mathcal I(E_0,E_1)$. Moreover,
    $$\|TXS\|_{\mathcal I}\le \|T\|_{L(E_1,E_2)}\|X\|_{\mathcal I}\|S\|_{L(E,E_0)} .$$
    \item If $x'\in E_0^*$ and $y\in E_1$ then $\|x'(\cdot)y\|_{\mathcal I}=\|x'\|_{E_0^*}\|y\|_{E_1},$ where $(x'(\cdot))y (x):=x'(x)y$ for $x \in E_0.$
\end{enumerate}

The symbols $\HH, \mc{E}, \mc{F}$ denote complex or real Hilbert spaces, $L(\HH)^{+}$ is the cone of semidefinite positive operators and $\leq$ stands for the order in $L(\HH)$  induced by $L(\HH)^+$, i.e., given $A,B \in L(\HH)$, $A \leq B$ if $B-A \in L(\HH)^+.$

Given $x\in \HH$ and $y \in \mc{F}$ the operator $\PI{\cdot}{x}y : \HH \rightarrow \mc{F}$ is defined by $(\PI{\cdot}{x}y) \ h:=\PI{h}{x} y,$ for $h \in \HH.$ Note that $(\PI{\cdot}{x}y)^*=\PI{\cdot}{y}x$ and that if $W \in L(\mc{F},\mc{E})$ then $W \PI{\cdot}{x}{y}=\PI{\cdot}{x}Wy.$ Recall also that $\tr(\PI{\cdot}{x}y)=\PI{x}{y},$ where $\tr$ denotes the trace of an operator.

Let $T\in L(\HH, \mc{F})$ be a compact operator. By $\{\lambda_k(T)\}_{k\geq1}$ we denote the eigenvalues of $\vert T \vert := (T^{*}T)^{1/2}\in L(\HH),$ where each eigenvalue is repeated according to its multiplicity. We say that $T$ belongs to the $2$-Schatten class $S_2,$ if $\sum_{k\geq1}^{} \lambda_{k}(T) ^{2}<\infty$ and, the $2$-Schatten norm is given by 
$\Vert T\Vert_2:= (\sum_{k\geq1}^{} \lambda_{k}(T) ^{2})^{1/2}.$ 
Recall that $S_2$ is a normed ideal of operators on Hilbert spaces. 
The reader is referred to \cite{ringrose1971compact,weidmann2012linear} for further details on these topics.

%

\vspace{0,5cm}
The Fr\'echet derivative will be instrumental to prove some results. We recall that, for a Banach space $(E, \Vert \cdot \Vert)$ and an open set $\mathcal U \subseteq E,$ a function $f: E \rightarrow \mathbb{R}$ is said to be \emph{Fr\'echet differentiable} at $X_0 \in \mathcal U$ if there exists $Df(X_0)$ a bounded linear functional such that
$$\lim\limits_{Y\rightarrow 0} \frac{|f(X_0+Y)-f(X_0) - Df(X_0)(Y)|}{\Vert Y \Vert}=0.$$ 
If $f$ is Fr\'echet differentiable at every $X_0 \in E$, $f$ is called Fr\'echet differentiable  on $E$  and the function $Df$ which assigns to every point $X_0 \in E$ the derivative $Df(X_0),$ is called the Fr\'echet derivative of the function $f.$ If, in addition, the derivative $Df$ is continuous, $f$ is  said to be a \emph{class $\mc{C}^1$-function}, in symbols, $f \in \mc{C}^1(E, \mathbb{R}).$

%

\begin{proposition} \label{TeoD} Given $x_0 \in \mathcal H$ and $W_1, W_2 \in L(\mathcal F)^+$ such that $W_1^{1/2} \in S_2.$  Let $K, k: L(\HH, \mc{F}) \rightarrow \mathbb{R}$ be defined by $K(X)=\Vert W_1^{1/2}X \Vert_2^{2}$ and  $k(X)=\PI{W_2Xx_0}{Xx_0}.$ Let $X, Y \in L(\HH, \mc{F})$ then, $K$ and $k$ have  Fr\'echet derivatives given by 
$$
DK(X)(Y)=2 \ Re \ [ \tr ( X^*W_1Y)],
\qquad
Dk(X)(Y)=2 \ Re \PI{W_2Xx_0}{Yx_0}.
 $$
 
\end{proposition}
See \cite[Theorem 2.1]{aiken1980unitary}, for a related result.

\section{The Total Least Squares problem on infinite dimensional spaces.}\label{section TLS}




Given $(\mathcal I, \| \cdot \|_\mathcal {I})$ a  normed  ideal of operators and let $(E_0,\|\cdot\|_{E_0})$, $(E_1,\|\cdot\|_{E_1})$ be normed spaces. Consider $A \in L(E_0,E_1), W_0 \in \mathcal I(E_1), W_1\in L(E_1)$ and $b\in E_1$, the problem is to determine if there exists 
\begin{equation} \label{TLSP}
\min_{X\in L(E_0,E_1), \ x \in E_0} \|W_0(A-X)\|_{\mathcal I}^2+\|W_1(Xx-b) \|_{E_1}^2.
\end{equation}
We will refer to  problem \eqref{TLSP} as the {\it Total Least Squares Problem}.

We will say that problem \eqref{TLSP} is {\it trivial} when the minimum in \eqref{TLSP}
is 0.
Observe that problem \eqref{TLSP} is trivial if and only if $b\in R(A)+N(W_0)+ N(W_1).$
In fact, if $b=Ax_0+z_0+z_1,$ with $z_j\in N(W_j),$ $j=1,2,$ then taking $X_0=A+Z$, where $R(Z)\subseteq N(W_0)$ and $Zx_0=z_0$, it follows that the pair $(X_0,x_0)$ is a solution of problem \eqref{TLSP} and  the minimum in \eqref{TLSP}
is 0. 

Conversely, if the minimum in \eqref{TLSP} 
is 0. Then, there exists $X_0 \in L(E_0,E_1)$ and $x_0 \in E_0$ such that, $W_0(A-X_0)=0$ and $W_1(X_0x_0-b)=0.$ Then, $X_0=A+Z,$ for some $Z \in L(E_0,E_1)$ such that $R(Z)\subseteq N(W_0)$ and $b=X_0x_0+z_1$ for some $z_1 \in N(W_1).$ Hence, $b=Ax_0+Zx_0+z_1 \in R(A)+N(W_0)+N(W_1).$

In the next proposition, we suppose that $b \not \in R(A)+N(W_0)+ N(W_1).$
Then the infinite dimensional extension of problem \eqref{TLSP} never has solution:

\begin{proposition}\label{TLS no tiene sol} Let $A\in L(E_0, E_1), W_0 \in \mathcal I(E_1), W_1\in L(E_1)$ and $b\in E_1$ for some normed ideal of operators $\mathcal I$ such that  $W_0A$ is not bounded below. Then problem \eqref{TLSP} does not have solution.
\end{proposition}

\begin{proof}
Since $b\notin R(A)+N(W_0)+N(W_1)$. Then, $$\|W_0(A-X)\|^2_{\mathcal I}+\|W_1(Xx-b) \|_{E_1}^2>0,$$ for any $(X,x)\in L(E_0,E_1)\times E_1.$ 

Since $W_0A$ is not bounded below, for any $\varepsilon>0$, there is some $x\in E_0$ such that $\|x\|=1$ and $\|W_0Ax\|_{E_1}<\varepsilon$. 
	
Take $x'\in E_0^*$ such that $x'(x)=\|x'\|_{E_0^*}=1$ and define
$$
X_0=A+\varepsilon x'(\cdot) (b-Ax/\varepsilon).
$$
Then $X_0x/\varepsilon-b=Ax/\varepsilon-b+ x'(x) (b-Ax/\varepsilon)=0$ and
\begin{align*}
	\|W_0(X_0-A)\|_{\mathcal I} &=\| W_0\varepsilon x'(\cdot) (b-Ax/\varepsilon)\|_{\mathcal I} \\ &=\varepsilon\|x'\|_{E_0^*}\|W_0b-W_0 Ax/\varepsilon\|_{E_1} \le \varepsilon(\|W_0b\|_{E_1}+1).
	\end{align*}
	Therefore, $\|W_0(A-X_0)\|_{\mathcal I}^2+\|W_0(X_0x/\varepsilon-b) \|_{E_1}^2\le \varepsilon^2(\|W_0b\|_{E_1}+1)^2.$
	Since this is true for arbitrarily small $\varepsilon>0$ we conclude that problem \eqref{TLSP} does not have a solution.
\end{proof}

\begin{remark}
We may also formulate the total least squares problem  imposing that the variable operator is in the normed ideal. In this case, no assumption on the weight is necessary in order to pose the problem (in particular, the problem without weights is also possible). That is, given $A\in \mathcal I(E_0,E_1),$ $b \in E_1$ and $W_0 , W_1\in L(E_1)$, we may consider the problem of determining if there exists 
\begin{equation*} \label{TLS2}
\min_{X\in \mathcal I(E_0,E_1), \, x \in E_0}
\|W_0(A-X)\|_{\mathcal I}^2+\|W_1(Xx-b) \|_{E_1}^2.
\end{equation*}
In this case, the same reasoning as in Proposition \ref{TLS no tiene sol} shows that if $W_0A$ is not bounded below then this problem does not have solution. Note that if $\mathcal I$ is an ideal of compact operators then $W_0 A$ is never bounded below.
\end{remark}


	

In particular, let $\HH$ and $\mc{F}$ be real or complex Hilbert spaces and, consider $A\in L(\mathcal H, \mc{F}),$ $b \in \mathcal F$ and $W_0=W_1=W \in L(\mathcal F)^+$ such that $W^{1/2} \in S_2$. The total least squares problem in this case is to determine if there exists 

\begin{equation}\tag{WTLS} \label{TLS}
\min_{X\in L(\mathcal H,\mathcal F), \ x \in \mathcal H} \|A-X\|_{2,W}^2+\|Xx-b \|_W^2.
\end{equation}

We will refer to this problem as \textit{weighted total least squares problem in Hilbert spaces}.
Observe that problem \eqref{TLS} is trivial if and only if $b\in R(A)+N(W).$
For the rest of this section, we suppose that $b\not \in R(A)+N(W).$
\begin{corollary} Let $A\in L(\mathcal H, \mc{F}),$ $b \in \mathcal F$ and $W \in L(\mathcal F)^+$ such that $W^{1/2} \in S_2.$ Suppose that $dim(\mathcal H)=\infty.$ Then problem \eqref{TLS} does not have a solution.
\end{corollary}
\begin{proof}
	The result follows from the above proposition because  since $W^{1/2}$ is a Hilbert-Schmidt (compact) operator, $W^{1/2}A$ cannot be bounded below.
\end{proof}

\begin{corollary} Let $A\in L(\mathcal H, \mc{F}),$ $b \in \mathcal F$ and $W\in L(\mathcal F)^+$ such that $W^{1/2} \in S_2.$ If $A$ is not inyective then problem \eqref{TLS} 
does not have a solution.
\end{corollary}

\section{Regularized weighted total least squares on Hilbert spaces} \label{RWTLS}
From the previous section we know that the  total least squares problem on infinite dimensional Banach spaces does not have solution unless we are in a trivial case. In this section we consider an associated problem, namely the Tikhonov regularized problem in infinite dimensional Hilbert spaces. 
We show some necessary and some sufficient conditions for the Tikhonov regularized problem to have solution and we present an example of existence. 

Let $\HH, \mc{F},\mc{E}$ be real or complex Hilbert spaces, $A\in L(\mathcal H, \mc{F})$, $T\in L(\HH,\mc{E}),$ $b \in \mathcal F$ and $W \in L(\mathcal F)^+$ such that $W^{1/2} \in S_2.$
Consider the following problem: finding the set of solutions of
\begin{equation} \tag{RWTLS} \label{RTLS}
\underset{X\in L(\HH, \mc{F}), \ x\in \HH}{\min }(\|Tx\|^2+\|A-X\|_{2,W}^2+\|Xx-b \|_W^2).
\end{equation} 
We will refer to problem \eqref{RTLS} as the \emph{regularized weighted total least squares problem}.

 \bigskip
In case $(A_0,x_0)$ is a solution of problem \eqref{RTLS} and  $\delta= \Vert Tx_0 \Vert,$  it is easy to see that $(A_0,x_0)$ is a solution of the \ref{TLS} problem with a quadratic constraint, i.e., $(A_0,x_0)$ is a solution of
\begin{equation*}
\min_{X\in L(\mathcal H,\mathcal F), \ x \in \mathcal H, \ \Vert T x \Vert \leq \delta} \ \ \|A-X\|_{2,W}^2+\|Xx-b \|_W^2.
\end{equation*}


\bigskip
Note that the minimum in \eqref{RTLS} is 0 if and only if $b\in A(N(T)) +N(W).$ We will say that in this case problem \eqref{RTLS} is {\emph trivial}.

As in the total least squares problem, there are many cases where the \eqref{RTLS} problem does not have a solution. 

\begin{proposition}\label{RTLS no tiene sol}
	Let $A\in L(\mathcal H, \mc{F})$, $T\in L(\HH,\mc{E}),$ $b \in \mathcal F$ and $W \in L(\mathcal F)^+$ such that $W^{1/2} \in S_2.$ 
	If $T^*T +A^*WA$ is not bounded below (e.g. if $W^{1/2}A$ is not bounded below on $N(T)$) then either the problem  \eqref{RTLS} is trivial  or it does not have a solution.
\end{proposition}
\begin{proof}
	Since  $T^*T +A^*WA$ is not bounded below, then it is not difficult to see that  $(T^*T +A^*WA)^{1/2}$ is not bounded below. Therefore, given $\varepsilon >0$ there exists $x\in \HH$ such that $\|x\|=1$ and
	$$
	\| (T^*T +A^*WA)^{1/2} x\| < \varepsilon^2.
	$$
	Observe that $	\| Tx \|^2=\langle T^*T x,x\rangle\leq \langle (T^*T +A^*WA)x,x\rangle =	\| (TT^* +A^*WA)^{1/2}x \|^2. $ Similarly $	\| W^{1/2}A x\|^2\leq\| (TT^* +A^*WA)^{1/2}x \|^2. $
	
	 Consider 
	 $$
	 X_0=A+\langle\cdot,\varepsilon x\rangle (b-Ax/\varepsilon).
	 $$
	 Then 
 $ X_0x/\varepsilon-b=Ax/\varepsilon+\|x\|^2(b-Ax/\varepsilon)-b=0
 $
 and 
 $$
 \|A-X_0\|_{2,W}^2=\|\langle\cdot,\varepsilon x\rangle (b-Ax/\varepsilon)\|_{2,W}^2=\|\langle\cdot,\varepsilon x\rangle W^{1/2} (b-Ax/\varepsilon)\|_{2}^2.
 $$
 Hence
  $$
  \begin{array}{lll}
  \|Tx/\varepsilon\|^2+\|A-X_0\|_{2,W}^2+\|X_0x/\varepsilon-b \|_W^2 &=&  \|Tx/\varepsilon\|^2+ \|\langle\cdot,\varepsilon x\rangle W^{1/2} (b-Ax/\varepsilon)\|_{2}^2\\
  &=& \frac{\|Tx\|^2}{\varepsilon^2}+ \|\langle\cdot,\varepsilon x\rangle W^{1/2} (b-Ax/\varepsilon)\|_{2}^2\\
  &\leq& \frac{\|Tx\|^2}{\varepsilon^2}+ \varepsilon^2\| x\|^2\| W^{1/2} (b-Ax/\varepsilon)\|^2\\
    &\leq& \frac{\|Tx\|^2}{\varepsilon^2}+ \varepsilon^2(\| W^{1/2} b\| +\|W^{1/2}Ax/\varepsilon\|)^2\\
   &\leq& \frac{\varepsilon^4}{\varepsilon^2}+ \varepsilon^2(\| W^{1/2} b\| +\varepsilon^2)^2\\
    &\leq&\varepsilon^2(1+(\| W^{1/2} b\| +\varepsilon^2)^2).
  \end{array}
  $$
  Since this holds for arbitrary $\varepsilon>0$, either there exists $x\in N(T)$ such that $W^{1/2}Ax=W^{1/2}b$ or,
  the problem \eqref{RTLS} does not have a solution.
	 \end{proof}

The following simple example shows that we may have a \eqref{RTLS} problem such that 
any finite dimensional restriction of the problem has a solution but the problem itself does not.

\begin{example}\label{ejemplo intro}
Let $\mathcal H=\mathcal F=\mathcal E$, let $A$ be any operator, $b\ne0$,  $W^{1/2}\in\mathcal S_2$ injective and $T$ any injective compact operator. Then by the above proposition, problem \eqref{RTLS} does not have a solution.
   	
On the other hand, suppose that  $\M,\N\subset \mathcal H$ are any finite dimensional subspaces such that $A(\M)\subseteq \N$, and  consider  the truncated \eqref{RTLS} problem, that is, the problem is to find the set of solutions of
\begin{equation*}
\underset{X\in L(\M, \N), \ x\in \M}{\min }(\|T|_\M x\|^2+\|A|_\M-X\|_{2,W}^2+\|Xx-b\|_W^2).
\end{equation*} 
Since $T|_\M$ and $W^{1/2}|_\N$ are bounded below, then  $\|T|_\M x\|^2+\|A|_\M-X\|_{2,W}^2+\|Xx-b\|_W^2$ is  a coercive continuous function on the finite dimensional space $L(\M, \N) \times \M$. Therefore, the truncated \eqref{RTLS} has solution.
\end{example}

\begin{remark}
Suppose that $(A_0, x_0)$ is a solution of problem \eqref{RTLS}. Then $x_0$ is a solution of the smoothing (regularized least squares) problem:  
\begin{equation} \label{RTLSmoth}
\underset{x\in \HH}{\min} \ (\|Tx\|^2+\|A_0x-b \|_W^2).
 \end{equation} 
 On the other hand, if $F:L(\HH, \mc{F})\rightarrow \mathbb R$ is defined as
 $F(X):=\|A-X\|_{2,W}^2+\|Xx_0-b \|_ W^2$,  then $A_0$ minimizes $F.$
\end{remark}

We now study some necessary conditions for the existence of solution of problem \eqref{RTLS}.
\begin{proposition}\label{CPO RTLS}
Let $A\in L(\mathcal H, \mc{F})$, $T\in L(\HH,\mc{E}),$ $b \in \mathcal F$ and $W \in L(\mathcal F)^+$ such that $W^{1/2} \in S_2.$  Suppose that $(A_0, x_0)$ is a solution of problem \eqref{RTLS} then $x_0$ is a solution of the normal equation
	\begin{equation}\label{CPO RTLS v1}
	T^*Tx+A_0^*W(A_0x-b)=0.
	\end{equation}
	Moreover,  $WA_0$ is the following rank one perturbation of $WA$
	\begin{equation}\label{CPO RTLS v2}
	WA_0=WA-\langle\cdot,x_0\rangle W(A_0x_0-b).
	\end{equation}
	In particular, $A_0^*W(A_0-A)=0$.
\end{proposition}
\begin{proof}

Suppose that $(A_0, x_0)$ is a solution of problem \eqref{RTLS}. Then, by the last remark, $x_0$ is a solution of problem \eqref{RTLSmoth}. Then, $x_0$ is a solution of the normal equation \eqref{CPO RTLS v1}; see \cite[proof of Theorem 3.2]{corach2016optimal}. 
	
On the other hand, by the last remark, since $A_0$ minimizes $F$ and $F \in \mc{C}^1(L(\HH,\mc{F}),\mathbb{R}),$  $A_0$ must be a critical point of $F.$
By Proposition \ref{TeoD}, it is not difficult to see that
\begin{align*}
	DF(X)(Y)&=2Re \big(\tr((W^{1/2}(X-A))^*W^{1/2}Y)\big)+2Re\big(\PI{W(Xx_0-b)}{Yx_0}\big)\\
	&=2Re\big( \tr ((X-A)^*WY)\big)+2Re\big(\PI{Y^*W(Xx_0-b)}{x_0}\big).
\end{align*}
Then, $DF(A_0)(Y)=0$, for every $Y\in L(\HH, \mc{F})$. So that 
\begin{align*}
	0&=Re\big( \tr ((A_0-A)^*WY)\big)+Re\big(\PI{x_0}{Y^*W(A_0x_0-b)}\big)=\\
	&=Re\big( \tr (Y^*W(A_0-A))\big)+Re\big(\tr(\PI{\cdot}{x_0}Y^*W(A_0x_0-b))\big)=\\
	&=Re\big( \tr (Y^*W(A_0-A))\big)+Re\big(\tr(Y^*\PI{\cdot}{x_0}W(A_0x_0-b))\big)=\\
	&=Re\big( \tr (Y^*[W(A_0-A)+\PI{\cdot}{x_0}W(A_0x_0-b)]\big).
\end{align*}
Thus 
	$$
	W(A_0-A)+\langle \cdot,x_0 \rangle W(A_0x_0-b)=0.
	$$
\end{proof}

\begin{corollary} \label{cor1}Let $A\in L(\mathcal H, \mc{F})$, $T\in L(\HH,\mc{E}),$ $b \in \mathcal F$ and $W \in L(\mathcal F)^+$ such that $W^{1/2} \in S_2.$ Suppose that $(A_0, x_0)$ is a solution of problem \eqref{RTLS}. Then 
	$$
	(1+\|x_0\|^2)T^*Tx_0+A^*W(Ax_0-b) = \frac{\|Ax_0-b\|^2_{W}}{1+\|x_0\|^2}x_0.
	$$ 
	Moreover,
	$$
	WA_0=WA-
	\langle \cdot,x_0 \rangle W \frac{Ax_0-b}{1+\|x_0\|^2}.
	$$
\end{corollary}

\begin{proof}	
	Note that by \eqref{CPO RTLS v2}, $W(A_0x_0-b)=WAx_0-\|x_0\|^2W(A_0x_0-b)-Wb$ and thus, 
	\begin{equation} \label{eq5}
	(1+\|x_0\|^2)W(A_0x_0-b)=W(Ax_0-b).
	\end{equation} Therefore, 
\begin{align*}
\|Ax_0-b\|^2_{W}&=\PI{W(Ax_0-b)}{Ax_0-b}=(1+\|x_0\|^2)\PI{W(A_0x_0-b)}{Ax_0-b}\\
&=(1+\|x_0\|^2)^2\PI{A_0x_0-b}{W(A_0x_0-b)}=(1+\|x_0\|^2)^2\|A_0x_0-b\|^2_{W},
\end{align*}
and consequently, 
	\begin{align*}
	A^*W(Ax_0-b) &= A^*W(A_0x_0-b)(1+\|x_0\|^2) \\
	&=\big(A_0^*W+\langle \cdot,W(A_0x_0-b) \rangle x_0\big)(A_0x_0-b)(1+\|x_0\|^2) \\
	&=
	\big(A_0^*W(A_0x_0-b)+\|A_0x_0-b\|^2_{W}x_0\big)(1+\|x_0\|^2) \\
	&=-(1+\|x_0\|^2)T^*Tx_0+\frac{\|Ax_0-b\|^2_{W}}{1+\|x_0\|^2}x_0, 
	\end{align*}
	where we used \eqref{CPO RTLS v1} for the last equality. Finally, by \eqref{CPO RTLS v2} and \eqref{eq5}, it follows that $
	WA_0=WA-
	\langle \cdot,x_0 \rangle W \frac{Ax_0-b}{1+\|x_0\|^2}.
	$
\end{proof}

Inspired by the results in \cite{beck2006solution,golub1999tikhonov} for finite dimensional spaces, we prove that the \eqref{RTLS} problem has a solution $(A_0,x_0)$ if and only if, $x_0$ minimizes some one-variable function.

\begin{theorem}\label{RTLS simpl} Let $A\in L(\mathcal H, \mc{F})$, $T\in L(\HH,\mc{E}),$ $b \in \mathcal F$ and $W \in L(\mathcal F)^+$ such that $W^{1/2} \in S_2.$ Let $x \in \HH$ and consider $F_x: L(\HH,\mathcal F) \rightarrow \mathbb{R},$
$$
F_x(X)=\|Tx\|^2+\|A-X\|_{2,W}^2+\|Xx-b \|_W^2.
$$
Then,  for every $x\in\HH,$ there exists $A_x \in L(\HH,\mathcal F)$ which minimizes $F_x.$

Moreover $$F_x(A_x)=\frac{\|Ax-b\|_W^2}{1+\|x\|^2}+\|Tx\|^2=:G(x).$$

\end{theorem}
\begin{proof}	
Fixed $x\in\HH$, proceeding as in the proof of Proposition \ref{CPO RTLS}, if $DF_x(X)(Y)=0$ for every $Y.$ Then 
\begin{equation}\label{CPO Fx}
W(X-A)+\langle \cdot,x \rangle W(Xx-b)=0.
\end{equation}

We claim that $A_x$ verifies the first order conditions \eqref{CPO Fx} if and only if $F_x$ has a minimum in $A_x.$ In fact, suppose that $A_x$ satisfies \eqref{CPO Fx}.
 Then \begin{align*}
F_x(X)&=\Vert Tx \Vert + \|A-A_x\|_{2,W}^2+\|A_x-X \|_{2,W}^2+2 Re \big(  \tr[(A_x-X)^*W(A-A_x)]\big)+\|A_xx-b\|_{W}^2 \!\!\! \\
 &+\|Xx-A_xx \|_W^2+2 Re \big( \langle W(A_xx-b), Xx-A_x x \rangle \big)\\
&=F_x(A_x )+\|A_x-X \|_{2,W}^2+\|Xx-A_x x \|_W^2,
\end{align*}
where the second equality follows because 
\begin{align*}
\tr[(A_x-X)^*W(A-A_x)]&=\tr[(A_x-X)^*\langle \cdot,x \rangle W(A_xx-b)]\\
&=\langle W(A_xx-b), (A_x-X)x\rangle=-\langle W(A_xx-b), Xx-A_x x\rangle.
\end{align*}
Therefore, 
$F_x(X)\geq F_x(A_x)$ for every $X\in L(\HH, \mc{F})$, so that $A_x$ is a minimum of $F_x$. 
The converse follows from the fact that $F_x \in \mc{C}^1(L(\HH,\mc{F}),\mathbb{R}).$

Moreover, it is not difficult to see that  $A_x=A+\frac{\langle \cdot,x \rangle}{1+ \|x\|^2}(b-Ax)$ satisfies \eqref{CPO Fx}. Therefore, $F_x$ has a minimum.

If $A_x$ is a minimum of $F_x$ then, $A_x$ is a solution of equation \eqref{CPO Fx} and proceeding as in Corollary \ref{cor1}, it holds that 
$(1+\|x\|^2)^2\|A_xx-b\|^2_{W}=\|Ax-b\|^2_{W}$ and  $\|A-A_x\|_{2,W}^2=\|x\|^2\|A_x
x-b\|^2_{W}.$  Consequently, the minimum of $F_x$ is 
	\begin{align*}
	F_x(A_x) &=\underset{X\in L(\HH, \mc{F})}{\min }{F_x(X)}=  \|Tx\|^2+\frac{\|Ax-b\|_W^2}{(1+\|x\|^2)^2}+\|x\|^2\frac{\|Ax-b\|_W^2}{(1+\|x\|^2)^2}\\ &=\|Tx\|^2+\frac{\|Ax-b\|_W^2}{1+\|x\|^2}=:G(x).
	\end{align*}
\end{proof}

\begin{corollary} \label{cor RTLS simpl}  Let $A\in L(\mathcal H, \mc{F})$, $T\in L(\HH,\mc{E}),$ $b \in \mathcal F$ and $W \in L(\mathcal F)^+$ such that $W^{1/2} \in S_2.$ 
Then, there exists $A_0$ such that $(A_0,x_0)$ is a solution of problem \eqref{RTLS} if and only if $x_0$ is a minimum of $G.$ 

In this case, $A_0=A+\frac{\langle \cdot,x_0 \rangle}{1+ \|x_0\|^2}(b-Ax_0).$ 
\end{corollary}

\begin{proof} If $(A_0,x_0)$ is a solution of problem \eqref{RTLS} then, in one hand $F_{x_0}(A_0)\leq F_x(X),$ for every $x\in \HH$, $X\in L(\HH, \mc{F})$ and, by Theorem \ref{RTLS simpl}, $F_{x_0}$ has a minimum in $A_{x_0},$ so that $$
G(x_0)=F_{x_0}(A_{x_0}) =F_{x_0}(A_0)\leq F_x(X),
$$
for every $x\in \HH$, $X\in L(\HH, \mc{F})$. By Theorem \ref{RTLS simpl}, $F_x$ has a minimum for every $x \in \HH,$ then $$G(x_0)\leq \underset{X\in L(\HH, \mc{F})}{\min }{F_x(X)}=G(x)$$ for every $x\in \HH$. 
	
Conversely, if $x_0$ is a minimum of $G(x)$ then, by Theorem \ref{RTLS simpl},
	$$
	 \underset{X\in L(\HH, \mc{F})}{\min }{F_{x_0}(X)}=G(x_0)\leq G(x)= \underset{X\in L(\HH, \mc{F})}{\min }{F_x(X)}\leq F_x(X),
	$$
		for every $x\in \HH$, $X\in L(\HH, \mc{F})$. As proved in Theorem \ref{RTLS simpl}, $A_0:=A+\frac{\langle \cdot,x_0 \rangle}{1+ \|x_0\|^2}(b-Ax_0)$ is a minimum of $F_{x_0}$ and, 
		$$
		F_{x_0}(A_0)=\underset{X\in L(\HH, \mc{F})}{\min }{F_{x_0}(X)}=G(x_0)\leq  F_x(X),
		$$
		for every $x\in \HH$, $X\in L(\HH, \mc{F})$. Therefore, $(A_0,x_0)$ is a solution problem \eqref{RTLS}.
\end{proof}

\begin{remark}
	In finite dimensional Hilbert spaces, if $T$ is invertible it is known that $G$ has always a minimum because $G$ is coercive, therefore \eqref{RTLS} has a solution. See \cite[Section 3]{beck2006solution}.
\end{remark}

Until now we have seen conditions that a solution of the \eqref{RTLS} must satisfy, but the only cases in the infinite dimensional setting we presented do not have solution (Proposition \ref{RTLS no tiene sol}).  
We now give an example of a diagonal operator on an infinite dimensional space.

\begin{example}\label{ejemplo diagonal}
	Let $\mathcal H=\mathcal F=\mathcal E=\ell_2$, the real Hilbert space of square summable sequences. Let $A$ be the diagonal operator $Ax=(a_nx_n)_n$, $b=\sum_{j=1}^Nb_je_j$ a finite sequence, where $(e_j)_j$ is the cononical basis and $W$ a diagonal weight operator, with weights $(w_n)_n$ in the diagonal. Let us see that  problem \eqref{RTLS}  has a solution for $T=\rho I$, for every $\rho>0$.
	
	Observe that, for any $\alpha=\sum_{j=1}^N\alpha_je_j$ and any $s=\sum_{j>N}s_je_j$,
	$$
	G(\sum_{j=1}^N\alpha_je_j+\sum_{j>N}s_je_j)=\frac{\sum_{j=1}^Nw_j(a_j\alpha_j-b_j)^2+\sum_{j>N}w_ja_j^2s_j^2}{1+\|\alpha\|^2+\|s\|^2}+\rho^2(\|\alpha\|^2+\|s\|^2).
	$$

	
	Thus, for any $\alpha,s$, 
	$$
	G(\sum_{j=1}^N\alpha_je_j+\sum_{j>N}s_je_j)\ge\frac{\sum_{j=1}^Nw_j(a_j\alpha_j-b_j)^2}{1+\|\alpha\|^2+\|s\|^2}+\rho^2(\|\alpha\|^2+\|s\|^2):=h(\alpha,\|s\|).
	$$
	Identifying the span of the $N$ first canonical vectors with $\mathbb R^N$, the function $h$ may be seen as a function from $\mathbb R^{N+1}$ to $\mathbb R.$
	Note that it suffices to prove that $h$ has a minimum that is attained at a point of the form $(\alpha^*,\|s^*\|)=(\hat\alpha,0)$. Indeed, if we prove it then for any $\alpha,s$ the following holds,
	$$
	G(\sum_{j=1}^N\alpha_je_j+\sum_{j>N}s_je_j)\ge h(\alpha,\|s\|)\ge h(\hat\alpha,0)=\frac{\sum_{j=1}^Nw_j(a_j\hat\alpha_j-b_j)^2}{1+\|\hat\alpha\|^2}+\rho^2(\|\hat\alpha\|^2)=G(\hat\alpha).
	$$ 
	In other words, $\hat\alpha$ would be a global minimum of $G$.

	Note also that $h$ is a coercive everywhere differentiable function of $N+1$ variables. Then, its minimum must be attained at a critical point. 
	Thus, it is sufficient to show that for any critical point $(\alpha,\|s\|)$, we have $h(\alpha,\|s\|)\ge\min_{\hat\alpha}h(\hat\alpha,0)$.

	 Since,
	$$
	\frac{\partial h}{\partial \|s\|}=-\frac{(\sum_{j=1}^Nw_j(a_j\alpha_j-b_j)^2)2\|s\|}{(1+\|\alpha\|^2+\|s\|^2)^2}+2\rho^2\|s\|,
	$$
	$\frac{\partial h}{\partial \|s\|}=0$ implies that either $s=0$   or
	\begin{align}\label{dh/ds=0}
	\rho^2=\frac{(\sum_{j=1}^Nw_j(a_j\alpha_j-b_j)^2)}{(1+\|\alpha\|^2+\|s\|^2)^2}.
	\end{align}  
	If $s=0$, we are done because the critical point is of the form $(\hat\alpha,0)$. 
	
	For the other case, since 
	$$
	\frac{\partial h}{\partial \alpha_j}=\frac{2w_ja_j(a_j\alpha_j-b_j)(1+\|\alpha\|^2+\|s\|^2)-2\alpha_j(\sum_{j=1}^Nw_j(a_j\alpha_j-b_j)^2)}{(1+\|\alpha\|^2+\|s\|^2)^2}+2\rho^2\alpha_j,
	$$
	we have that $\frac{\partial h}{\partial \alpha_j}=0$ and equation \eqref{dh/ds=0} imply that $\displaystyle \frac{w_ja_j(a_j\alpha_j-b_j)}{1+\|\alpha\|^2+\|s\|^2}=0$.
	
	Thus, if $\{1,...,N\}= \mathcal C\cup \mathcal D$ with $w_ia_i=0$ for $i\in \mathcal C $ and $w_ia_i\neq0$ for $i\in  \mathcal D$, then $\alpha_j=\frac{b_j}{a_j}$ for every $j \in  \mathcal D$. 
	
	 Suppose first that $w_ja_j\ne 0$ for every $j\le N$, i.e. $ \mathcal C=\emptyset$. Then $\alpha_j=\frac{b_j}{a_j}$ for every $j\le N$. Thus, replacing this  again in \eqref{dh/ds=0}, we obtain $\rho=0$, which is a contradiction (and thus only the case $s=0$ is possible).  
	
	
	If $ \mathcal C\ne\emptyset$, say $k\in \mathcal C$, let $(\alpha^*,\|s^*\|)$ be a critical point. Thus $\alpha^*=\sum_{j\in\mathcal C}\alpha_j^*e_j+\sum_{j\in\mathcal D}\frac{b_j}{a_j}e_j.$ and let 
	$$
	\hat\alpha:=(\|s^*\|^2+\sum_{j\in\mathcal C}(\alpha_j^*)^2)^{1/2}e_k+\sum_{j\in\mathcal D}\frac{b_j}{a_j}e_j.
	$$
	Then, since $\|(\hat\alpha,0)\|_{\mathbb R^{N+1}}^2=
	\|(\alpha^*,\|s^*\|)\|_{\mathbb R^{N+1}}^2$,
	it easy to check that $h(\hat\alpha,0)=
	h(\alpha^*,\|s^*\|)$, which is what we wanted to prove.
	
\end{example}

\subsection{The case $T$ is a multiple of the identity} \label{section case-I}

In this subsection, we find some sufficient conditions for the existence of solution when the regularization operator $T$ is a multiple of the identity.
Let $A\in L(\mathcal H, \mc{F})$, $T\in L(\HH,\mc{E}),$ $b \in \mathcal F$ and $W \in L(\mathcal F)^+$ such that $W^{1/2} \in S_2,$  by Corollary \ref{cor RTLS simpl}, to solve the \eqref{RTLS} problem, is equivalent to minimizing the function $G$. We suppose in this section that $T=\rho^{1/2} I$ (a multiple of the identity), so that, the problem is to minimize
$$
G(x)=\frac{\|Ax-b\|_W^2}{1+\|x\|^2}+\rho \|x\|^2,
$$
for a given constant $\rho>0$. 
We will apply to $G$ the Dinkelbach method, see \cite[Section 5.2]{beck2006solution} and \cite{dinkelbach1967nonlinear}.

Let us call $t^*\geq0$ to the infimum of $G(x)$, varying $x\in \HH$. 
Then the infimum of the expression 
\begin{align}\label{Dinkelbach 1}
\|Ax-b\|_W^2 +\rho\|x\|^4+(\rho-t^*)\|x\|^2-t^*,
\end{align}
is 0. Moreover, $x_0$ minimizes $G$ if and only if $x_0$ minimizes 
\eqref{Dinkelbach 1}
and in this case, the minimum of the expression in \eqref{Dinkelbach 1} equals $0$.

Let us define 
$$
\phi(t):=\inf_x \{ \|Ax-b\|_W^2 +\rho\|x\|^4+(\rho-t)\|x\|^2-t\}.
$$
Then $\phi$ is a decreasing function and thus it has at most one zero. Moreover, since $\phi(t^*)=0$ by definition, $t^*$ is the only root of $\phi$.

\begin{corollary}\label{rho>b}
	If  $\rho\ge t^*$ then  problem \eqref{RTLS} with $T=\rho^{1/2}I$ has a unique solution.
\end{corollary}
\begin{proof}
	By the above comments, the infimum in  \eqref{Dinkelbach 1} is 0. Moreover, if we have that $\rho\ge t^*$ then \eqref{Dinkelbach 1} is a strictly convex coercive function of $x$ and therefore it has a unique minimizer $x_0$. By the above comments, $x_0$ is the unique minimizer of $G$ and, by Corollary \ref{cor RTLS simpl}, $x_0$ must be the unique solution of problem \eqref{RTLS}.
\end{proof}

\begin{remark} Suppose  $\rho\ge\|b\|_W^2$ then problem \eqref{RTLS} with $T=\rho^{1/2}I$ has a unique solution. 
	 	Indeed,  note that since $G(0)=\|b\|_W^2$, we have that $t^*$ is always less than or equal to   $\|b\|_W^2$.
\end{remark}
\bigskip

\subsection*{How to find $t^*$}

In this subsection we give a characterization of $t^*$ and as a corollary we present more  sufficient conditions for the existence of solutions of problem \eqref{RTLS}. In this subsection, we suppose that $\HH$ is a real Hilbert space.

The following lemma, which is a partial extension of \cite[Theorem 1]{nguyen2019solving} to infinite dimensional spaces will be a crucial tool for the results in this subsection. See \cite{continoPolyak}.

\begin{lemma}\label{new S-lemma}
Let $f(x)=\langle Sx,x\rangle+\langle x,a\rangle+s$, with $S\in L(\HH)$ nonnegative, $a\in\HH$, $s\in\mathbb R$, let  $g(x)=\|x\|^2$ and let
$F:\mathbb R^2\to\mathbb R$ defined as
$$
F(z)=\langle \Theta z,z\rangle+\langle z,v\rangle -t,
$$	
where $\Theta$ is a real symmetric nonnegative $2\times 2$ matrix, $v=(v_1,v_2)\in\mathbb R^2$ and $t\in\mathbb R$. Then the following are equivalent:
\begin{itemize}
	\item[(i)]  $F(f(x),g(x))\ge 0$ for every $x\in\HH$.
	
	\item[(ii)] There exist $\alpha,\beta\in \mathbb R$ such that for every $x\in \HH$ and every $z=(z_1,z_2)\in\mathbb R^2$,
	$$
	F(z)+\alpha(f(x)-z_1)+\beta (g(x)-z_2)\ge0.
	$$
\end{itemize}
Moreover, if $S$ is not bounded below, 
 $\Theta=\left(\begin{array}{cc}
	0			 & 0  \\
	0		 &  \rho
	\end{array}\right)$ and $v_1>0$,
 then $\alpha$ and $\beta$ can be chosen nonnegative.
\end{lemma}

\medskip

The following gives a characterization of the infimum $t^*.$ See \cite{nguyen2019solving} for a similar result on finite dimensional spaces.
\begin{proposition}\label{SDP} Let $A\in L(\mathcal H, \mc{F})$, $b \in \mathcal F,$ $W \in L(\mathcal F)^+$ such that $W^{1/2} \in S_2$ and $\rho>0.$ 
The infimum $t^*$ of $G$ is the maximum of all $t\in \mathbb R$ such that there exist $\alpha,\beta\in \mathbb R$ such that  $C\in L(\HH\times\mathbb R^3)$  is a nonnegative operator, where $C$ is the 
operator defined as
\begin{align*}
C:=
\left(\begin{array}{cccc}
\alpha A^*WA+\beta I 			 & 0 & 0    & -\alpha A^*Wb \\
0								 & 0 & 0    & \frac{1-\alpha}{2}\\
0								 & 0 & \rho & \frac{\rho-t-\beta}{2}\\
\langle\cdot,-\alpha A^*Wb\rangle & \frac{1-\alpha}{2} & \frac{\rho-t-\beta}{2} & \alpha\|b\|_W^2-t
\end{array}\right)	.
\end{align*}
\end{proposition}
\begin{proof}
Let us denote $f(x)=\|Ax-b\|_W^2$  and $g(x)=\|x\|^2$. Then note that
\begin{align*}
t^*=& \max_{t\in\mathbb R}\{t: f(x)+\rho g(x)^2+(\rho-t)g(x)-t\ge 0, \forall x\in\HH\} \\
=& \max_{t,\alpha,\beta\in\mathbb R}\{t: z_1+\rho z_2^2+(\rho-t)z_2-t +\alpha(f(x)-z_1)+\beta(g(x)-z_2)\ge 0,\forall x\in\HH, z_1,z_2\in\mathbb R\},\! \! 
\end{align*}
where the first equality holds by the comments at the beginning of the section and the  last equality holds from the above lemma applied to $f(x)=\|Ax-b\|_W^2=\langle A^*WAx,x\rangle-2\langle x,A^*Wb\rangle+\langle Wb,b\rangle$ and $F(z)= \rho z_2^2 +z_1+(\rho-t)z_2-t$. Note that, by the above lemma, $\alpha$ and $\beta$ can be chosen nonnegative.

Let $x\in\HH$, $z_1,z_2\in\mathbb R$, $y=(x,z_1,z_2,1)\in\HH\times\mathbb R^3$, then
\begin{align*}
\langle Cy,y\rangle &= \langle (\alpha A^*WA+\beta I)x,x\rangle-2\alpha  \langle  A^*Wb,x\rangle +(1-\alpha)z_1+\rho z_2^2+(\rho-t-\beta)z_2+\alpha\|b\|_W^2-t \\
& = z_1+\rho z_2^2+(\rho-t)z_2-t +\alpha(\|Ax-b\|_W^2-z_1)+\beta(\|x\|^2-z_2).
\end{align*}
Therefore $\langle Cy,y\rangle\ge 0$ for every $y=(x,z_1,z_2,1)\in\HH\times\mathbb R^3$ if and only if $z_1+\rho z_2^2+(\rho-t)z_2-t +\alpha(f(x)-z_1)+\beta(g(x)-z_2)\ge 0;\textrm{ for every }x\in\HH,\,z_1,z_2\in\mathbb R$.

Finally, note that $\langle C(x,z_1,z_2,0),(x,z_1,z_2,0)\rangle=\alpha\|Ax\|^2_W+\beta\|x\|^2+\rho z_2^2 $ is  nonnegative because $\alpha$ and $\beta$ can be chosen nonnegative.
\end{proof}

Let $A\in L(\mathcal H, \mc{F})$, $b \in \mathcal F,$  $W \in L(\mathcal F)^+$ such that $W^{1/2} \in S_2$ and $\rho>0.$  The regularized least squares problem 
\begin{equation} \label{RTLSp}
	\min_{x \in \HH} \|Ax-b\|^2_W+\rho\|x\|^4,
\end{equation}
always has a solution, because the objective function is convex and coercive. Let $a^*$ be the minimum of \eqref{RTLSp}.	
\begin{corollary} Suppose that $a^*\le \rho$. Then problem \eqref{RTLS} with $T=\rho^{1/2}I$ has a unique solution.
\end{corollary}

\begin{proof}
 It suffices to show that if we take any $t>\rho$ then for every $\alpha,\beta\in\mathbb R$, the matrix $C$ from Lemma \ref{SDP} is not nonnegative.  Indeed, in this case by Lemma \ref{SDP}, $t^*\le\rho$ and by Corollary \ref{rho>b}, the \eqref{RTLS} problem has a unique solution.
	
	Let $y=(x,z_1,z_2,1)\in\HH\times\mathbb R^3$, then
	\begin{align*}
	\langle Cy,y\rangle  = (1-\alpha)z_1+\rho z_2^2+(\rho-t-\beta)z_2-t +\alpha\|Ax-b\|_W^2+\beta\|x\|^2.
	\end{align*}
	Note that if $\alpha\ne 1$ then we can always choose $z_1$ so that $\langle Cy,y\rangle <0$. Suppose $\alpha=1$. Then 
	\begin{align*}
	\langle Cy,y\rangle  = \rho z_2^2+(\rho-t-\beta)z_2-t +\|Ax-b\|_W^2+\beta\|x\|^2.
	\end{align*}
	Taking $z_2=-\frac{\rho-t-\beta}{2\rho}$ then 
  $\langle Cy,y\rangle  = \|Ax-b\|_W^2+\beta\|x\|^2-\frac{(\rho-t-\beta)^2}{4\rho}-t$. Maximizing in $\beta$, 
	\begin{align*}
 \langle Cy,y\rangle \le \|Ax-b\|_W^2+(\rho-t)\|x\|^2+\rho\|x\|^4-t.  
	\end{align*}	
	Let $\tilde x\in\HH$ such that $a^*=\|A\tilde x-b\|^2_W+\rho\|\tilde x\|^4.$ Since $a^* \le \rho,$ for $y=(\tilde x,z_1,-\frac{\rho-t-\beta}{2\rho},1)$ we have 
	\begin{align*}
	\langle Cy,y\rangle  \le \|A\tilde x-b\|_W^2+(\rho-t)\|\tilde x\|^2+\rho\|\tilde x\|^4-t\le (\rho-t)\|\tilde x\|^2+(\rho-t)<0.
	\end{align*}	
	Therefore $C$ is not nonnegative.  
\end{proof}

\section{The restricted regularized total least squares problem} \label{lower-semi}
Let $A\in L(\mathcal H, \mc{F})$, $T\in L(\HH,\mc{E}),$ $b \in \mathcal F$ and $W \in L(\mathcal F)^+$ such that $W^{1/2} \in S_2$ and consider problem \eqref{RTLS}. Suppose that the regularization operator $T$ is invertible. Then, in the finite dimensional case, the existence of solution of problem \eqref{RTLS} is guaranteed by the fact that the objetive function $\|Tx\|^2+\|A-X\|_{2,W}^2+\|Xx-b \|_W^2$ is continuous and coercive on $L(\HH, \mc F) \times \HH$, see e.g. \cite{van1991total}. A natural approach, in the infinite dimensional case would be to minimize  coercive and weakly continuous (or at least weakly lower-semicontinuous) functions.

Since the norm is weakly lower-semicontinuous on any normed space, the first two terms in the objective function of the \eqref{RTLS} problem are  weakly lower-semicontinuous. Note that, if the mapping $(X,x)  \mapsto  Xx$ is (jointly) weakly continuous on bounded sets, then the third term of the objective function $\|Xx-b\|_W^2$ is also weakly lower-semicontinuous.

In \cite{bleyer2013double}, a regularized total least squares problem is studied on infinite dimensional Hilbert spaces. There   an existence theorem  is proved but their proof assumes a crucial property, which as we will see, is not satisfied in many reasonable cases.
Their results follows from an assumption (assumption $(A1)$ in \cite{bleyer2013double}) which translates to the fact that, for example the bilinear mapping defined as
\begin{align*}
    B: \mathcal S_2\times \HH &\to  \HH \\
    (X,x) & \mapsto  Xx,
\end{align*}
is weak-to-norm continuous. But, this is not true: any orthonormal basis  $(e_n)_n$ is weakly null (weak convergent to $0$) in $\mathcal H$ and if $X_n=\langle \cdot, e_n\rangle e_1$  (a rank 1 operator defined on $\mathcal H$) then $(X_n)_n$ is also weakly null in $\mathcal S_2$; indeed if $K\in \mathcal S_2$ 
then $\langle X_n, K\rangle =\tr(\langle \cdot, e_n\rangle e_1 K^*)=\langle K^*e_1, e_n\rangle \to 0$. But $X_ne_n=e_1$ for every $n$, and thus $(X_ne_n)_n$ does not converge to zero in any topology. Therefore $B$ is not weak-to-weak continuous. The same example shows that the function $ (X,x)  \mapsto\|Xx\|^2$ is not weakly lower-semicontinuous.
Moreover, since $(X_n)_n$ also converges to 0 in the strong operator topology (SOT), $B$ is also not $SOT\times$weak to weak continuous.

In order to assure existence of solution we may restrict either the set of operators or the set of vectors to smaller sets which have some kind of compacity. The aim of this section is to show that this is a delicate problem. We present a restricted regularized total least squares problem in a general setting and show some cases in which we can assure the continuity of the bilinear mapping and hence the existence of solution. In the final subsection we show some very natural examples in which the bilinear mapping fails to be continuous.

\bigskip

\subsection{Restricted regularized total least squares problem}

The continuity of the bilinear mapping can be used to prove existence of a regularized total least squares problem when restricted to suitable sets.

Let $E_0,E_1,E_2$ be infinite dimensional Banach spaces, $\mathcal I\subset L(E_0,E_1)$ be any normed ideal of operators and $C\subset \mathcal I$ and $D\subset E_0$ closed convex subsets.

Given $A\in C\subset\mathcal I$, $T\in L(E_0,E_2)$  and $b \in E_1$,
we consider the following \emph{restricted regularized total least squares problem}: find the set of solutions of
\begin{equation} \tag{RRTLS} \label{restricted RTLS}
\underset{X\in C\subset \mathcal I, \ x\in D\subset E_0}{\min }f(\|Tx\|_{E_2},\|A-X\|_{\mathcal I},\|Xx-b \|_{E_1})
\end{equation} 
with $f:\mathbb R_{\ge 0}^3\to\mathbb R_{\ge 0}$ is any continuous, increasing and coercive function. For example, if $f(t_0,t_1,t_2)=t_0^2+t_1^2+t_2^2$, then the function we should minimize is the same as in the previous section.

The most simple situation is when one of the subsets, $C$ or $D$ is norm compact:
\begin{proposition}\label{restringido a compactos}
    Suppose that $T$ is bounded below and that either:
\begin{enumerate}
    \item $C$ is compact, $D=E_0$ with $E_0$ reflexive.
    
    \item $C=\mathcal I$ is a reflexive Banach space of operators 
    and $D$ is compact.
    
    \item 
     $C=\mathcal I=L(E_0,E_1)$ with $E_1$ is reflexive and $D$ is compact.
\end{enumerate}
Then the \eqref{restricted RTLS} problem admits solution.
\end{proposition}

\begin{proof}
Let $g(X,x):=f(\|Tx\|_{E_2},\|A-X\|_{\mathcal I},\|Xx-b \|_{E_1})$ and let $B: \mathcal{I} \times E_0 \to  E_1, B(X,x)=Xx.$
Since $f$ is coercive and $T$ is bounded below, $g$ is also coercive.  Thus we may restrict the minimization problem to $\tilde D$ and $\tilde C$, the intersection of $D$ and $C$ with some closed balls, respectively.
\begin{enumerate}
\item The function
    $(X,x) \mapsto  Xx$ is continuous from $(\tilde C,\|\cdot\|_{\mathcal I})\times(\tilde D,w)$ to $(E_1,w)$.  In fact, let $y'\in E_1^*$, let $(X_n)_n\subset \tilde C$ be a norm convergent sequence to $X$ and  $(y_\lambda)_\lambda\subset \tilde D$ be a weak convergent net to $y$.
          Then
    \begin{align*}
    |y'(B(X_n, y_\lambda)-B(X,y))| & \le |y'(B(X_n-X, y_\lambda))| +|y'(B(X, y_\lambda-y))|. 
    \end{align*}
    The first term tends to zero because $X_n\overset{\|\cdot\|}{\longrightarrow}X$, $(y_\lambda)_\lambda$ is bounded and $B$  is norm bounded. 
    The second term approaches to zero because $y'(B(X,\cdot))$ is a continuous linear functional on $E_0$ and $y_\lambda\overset{w}{\to}y$.
    
        Thus, since the norm is a weakly-lower semicontinuous function, the composition $ (X,x) \mapsto  \|Xx-b\|_{E_1}$ is  weakly-lower semicontinuous. Similarly, the function $x\mapsto\|Tx\|_{E_2}$ is weakly-lower semicontinuous because $T$ is weak to weak continuous and $X\mapsto\|A-X\|_{\mathcal I}$ is continuous. Thus $g$ is weakly-lower semicontinuous on $(\tilde C,\|\cdot\|_{\mathcal I})\times(\tilde D,w)$, see \cite[Lemma 1.7]{penot1982semi}.     
    Finally, since $E_0$ is reflexive, $(C,\|\cdot\|_{\mathcal I})\times(\tilde D,w)$ is compact and therefore $g$ attains its minimum. 
    \item 
    The proof is similar, using  that
    $ (X,x) \mapsto  Xx$ is continuous from $(\tilde C,w)\times(\tilde D,\|\cdot\|_{E_0})$ to $(E_1,w)$ and the compactness of $(\tilde C,w)\times(\tilde D,\|\cdot\|_{E_0})$.
    
    \item The proof is the similar,  proving  the continuity of $ (X,x) \mapsto  Xx$  from $(\tilde C,WOT)\times(\tilde D,\|\cdot\|_{E_0})$ to $(E_1,w)$. We must also use the fact that the closed  unit ball of $L(E_0,E_1)$ is WOT-compact when $E_1$ is reflexive.
\end{enumerate}
\end{proof}


  



Whenever $\mathcal I=L(E_0,E_1)$,  the following result allows us to prove the existence of solution of problem \eqref{restricted RTLS} when we ask $C$ a condition which is weaker than compacity, namely weak equicompacity.

The following definition was given in \cite{serrano2006weakly}:
\begin{definition}
A subset $C\subset L(E_0,E_1)$ is said to be  {\it weakly $w_0$-equicompact} if for every weakly null sequence $(y_n)_n\subset E_0$  there exists a subsequence $(y_{n_k})_k$ such that $(Xy_{n_k})_k$ converges weakly uniformly for $X\in C$ to 0.
\end{definition}
\begin{proposition} \label{equicompact}
    Suppose that $T$ is bounded below and that  $C\subset L(E_0,E_1)$ is a closed and convex set which is weakly $w_0$-equicompact set of operators and that $D=E_0$ is a reflexive Banach space.

Then the \eqref{restricted RTLS} problem admits solution.
\end{proposition}
\begin{proof}
    We first observe that in this case the bilinear mapping $B$ is $(C,WOT)\times(D,w)$ to $(E_1,w)$ continuous. 
     This is a direct consequence of \cite[Lemma 2.6]{serrano2006weakly}, which tells us that $S(x_n-x)\overset{w}{\longrightarrow}0$ uniformly for $S\in C,$ for any $x_n\overset{w}{\longrightarrow} x$.
     
     Now, if $g(X,x):=f(\|Tx\|_{E_2},\|A-X\|,\|Xx-b \|_{E_1})$, then $g$ is coercive.  Thus we may restrict the minimization problem to $C\times\tilde{D}$, where $\tilde{D}$ is the intersection of $D$ with some closed ball.
     Also, since the operator norm is $WOT$ lower semicontinuous, we may proceed as in $(1)$ of Proposition \ref{restringido a compactos} to show that $g$ is lower semicontinuous on $( C,WOT)\times(\tilde{D},w)$.  Since $C$ is closed and convex, it is $WOT$-compact. Therefore, $g$ attains its minimum on $ C\times D$.
\end{proof}

\begin{remark}
Note that the conclusion of Propositions \ref{restringido a compactos} and \ref{equicompact} remain true for arbitrary $T$ if we suppose additionally that $D$ is bounded. In particular, this gives us existence results for the restricted total least squares problem without regularization.
\end{remark}

We present now an example showing that the above result can be applied to assure the existence of solution of problem \eqref{restricted RTLS} on sets of triangular operators.
This example is similar to Example \ref{ejemplo diagonal}, here $b$ is allowed to be any vector in $\ell_2$, but we must restrict to a proper subsets of operators.
\begin{example}
	Let $\mathcal H=\mathcal F=\mathcal E=\ell_2$,  $A\in L(\ell_2)$, $b\in\ell_2$, $T \in L(\ell_2)$ bounded below.  We show that \eqref{restricted RTLS} has solution when we minimize on a set $C_N$ of operators which contains all operators with lower triangular  matrix representations: given $N\ge 0$ let 
	$$
C_N=\{X\in L(\ell_2)\,: \, \langle Xe_j,e_i\rangle =0 \textrm{ for }i<j+N\}. 
	$$
Since $g(X,x):=f(\|Tx\|_2,\|A-X\|,\|Xx-b \|_{2})$ is coercive, we may restrict $x$ and $X$ to some closed balls. 
Thus, by Proposition \ref{equicompact}, it suffices to see that for $r>0$, $r\overline{B_{\mathcal I}}\cap C_N$ is a weakly $w_0$-equicompact set of operators. By  \cite[Corollary 2.3]{serrano2007some} this can be proved if we show that for each $y\in \ell_2$, the sets $(r\overline{B_{\mathcal I}}\cap C_N^{*})y:=\{X^*y: X\in C_N,\,\|X\|\le r\}$
are relatively compact sets in $\ell_2$. This is easily seen applying a classical result of Fr\'echet (see e.g. \cite[Theorem 4]{hanche2010kolmogorov}), according to which it suffices to see that given $\varepsilon>0$, there is some $n$ such that for every $X^*\in r\overline{B_{\mathcal I}}\cap C_N^{*}$,
$$
\sum_{j>n}\langle X^*y,e_j\rangle^2<\varepsilon.
$$
Let $n$ be such that that $\sum_{j>n-N}y_j^2<\varepsilon/r^2$ and denote by $y^{n-N}$ the tail of $y$ so that $\|y^{n-N}\|^2<\varepsilon/r^2.$
 Let $X\in r\overline{B_{\mathcal I}}\cap C_N$, then
$$
\sum_{j>n}\langle X^*y,e_j\rangle^2= \sum_{j>n}\left(\sum_{l>n-N}y_l\langle X^*e_l,e_j\rangle\right)^2=
\sum_{j>n}\langle X^*y^{n-N},e_j\rangle^2 \le \|X^*\|^2\|y^{n-N}\|^2
<\varepsilon.
$$

\end{example}

\medskip

For $\mathcal I=K(E_0,E_1)$, the space of compact operators, we can assure the existence of solution of the \eqref{restricted RTLS} problem restricted to weakly compact sets whenever the space $E_0$ has the Dunford-Pettis property. To achieve this we prove that the bilinear mapping in this case is weakly sequentially continuous. We will actually see in the next proposition that the Dunford-Pettis property characterizes this continuity for the bilinear mapping. 

Recall that a Banach space $E_0$ is said to have the \textit{Dunford-Pettis property} if for
each Banach space $E_1,$ every weakly compact linear operator $S:
E_0 \rightarrow E_1$ is
completely continuous, i.e., $S$ takes weakly compact sets in $E_0$ onto norm
compact sets in $E_1$.  An important characterization for $E_0$ to have the Dunford-Pettis property is that for any weakly null sequences $(x_n)_n$ of $E_0$ and $(y_n')_n$ of the dual space $E_0^*$, the sequence $y_n'(x_n)$ converges to $0.$ See \cite[Theorem 1]{diestel1980survey}.

Some examples of spaces with the
Dunford-Pettis property are $C(K)$ spaces, $L^1(\mu)$-spaces and spaces whose duals
are either $C(K)$ or $L^1(\mu)$-spaces, spaces of analytic functions like $H^\infty$ or the disc algebra or the spaces  of  smooth  functions on the $n$-dimensional torus  $C^{(k)}(\mathbb T^n)$. 
We refer the reader to \cite{diestel1980survey, castillo1994dunford}.

\begin{proposition}\label{prop dunford-pettis}
    Let  $B: K(E_0,E_1) \times E_0 \to E_1$ be the bilinear mapping,
    $$
    B(X,y)=Xy.
    $$
    Then $B$ is weakly sequentially continuous (that is, $B$ sends weakly convergent sequences in $E_0$ and $K(E_0,E_1)$ to a weakly convergent sequence in $E_1$) if and only if $E_0$ is a Banach space with the Dunford-Pettis property.
    
    In this case, if
    $C\subset K(E_0,E_1)$ and  $D\subset E_0$ are a weakly-compact sets then the \eqref{restricted RTLS} problem admits solution.
\end{proposition}
\begin{proof}
Suppose first that $E_0$ lacks the Dunford-Pettis property. Then, there are weakly null sequences $(y_n)_n\subset E_0$ and $(y_n')_n\subset E_0^*$ such that $y_n'(y_n)\to 1$. Take any nonzero vector $y_1\in E_1$ and define $X_n(y)=y_n'(y)y_1$. Thus, $z(X_n^*(y'))=z(y_n')y'(y_1)\to 0$ for every $z\in E_0^{**}$ and every $y'\in E_1^*$. Thus $X_n\overset{WOT*}{\longrightarrow}0$.  By \cite[Corollary 3]{kalton1974spaces}, $(T_n)_n$ converges weakly to $0$.
    
    On the other hand, $B(X_n,y_n)=y_n'(y_n)y_1\to y_1\ne B(w-\lim X_n,w-\lim y_n)=B(0,0)=0$. Therefore, $B$ is not weakly sequentially continuous.

    Conversely, suppose now that $E_0$ has the Dunford-Pettis property. Since $D\subset E_0$ is weakly compact then for each sequence $(y_n)_n\subset D$, $y_n'(y_n)\to 0$ for every weakly null sequence $(y_n')_n$. See \cite[Proposition 2.1, Proposition 2.3]{aqzzouz2017dunford} and \cite{diestel1980survey}.
    
    Let $y_n\overset{w}{\longrightarrow} y$, $X_n\overset{w}{\longrightarrow} X$ and take $y'\in E_1^*$. Since weak convergence implies $WOT*$ convergence,  it is not difficult to see that $(y'\circ(X_n-X))_n$ is a weakly null sequence in $E_0^*$, and thus 
    $$
    y'\circ (X_n-X)(y_n)\to 0.
    $$
    Therefore $y'(B(X_n,y_n)-B(X,y))=y'(B(X_n,y_n)-B(X,y_n))+y'(B(X,y_n)-B(X,y))=y'\circ(X_n-X)(y_n)+(y'\circ X)(y_n-y)\to 0.$ Hence,  $B$ is weakly sequentially continuous.
    
    In this case, if $C\subset K(E_0,E_1)$ and  $D\subset E_0$ are weakly-compact sets then the \eqref{restricted RTLS} problem admits solution. In fact, 
    since the norm is a weakly-lower semicontinuous function, the composition $ (X,x) \mapsto  \|Xx-b\|_{E_1}$ is sequentially weakly-lower semicontinuous. Then, if $g(X,x)=f(\|Tx\|_{E_2},\|A-X\|,\|Xx-b \|_{E_1})$, then $g$ is sequentially weakly-lower semicontinuous on $(C,w)\times(D,w)$, see \cite[Lemma 1.7]{penot1982semi}.     
    Finally, since $C$ and $D$ are weakly compact, by the Eberlein-Smulian Theorem, they are sequentially weakly compact and therefore $g$ attains its minimum.
\end{proof}

In the previous theorem we can drop the hypothesis that  $X \in K(E_0,E_1)$ to prove that $B$ is weakly sequentially continuous. That is, consider 
$B: L(E_0,E_1) \times E_0 \to E_1$ the bilinear mapping,
    $$
    B(X,y)=Xy,
    $$
and $E_0$ a Banach space with the Dunford-Pettis property. Then, for any weakly-compact sets
$C\subset L(E_0,E_1)$ and  $D\subset E_0$, the \eqref{restricted RTLS} problem admits solution. Note also that since $T$ need not to be bounded below here, we can also conclude the existence of solution for the restricted total least squares problem without regularization.

\subsection{A variant to the regularized total least squares problem}

We now briefly present a modification to the regularized total least squares problem which gives rise to a weak  continuous bilinear mapping. 

Let $E_0,E_1,E_2,E_3$ be infinite dimensional Banach spaces, $\mathcal I\subset L(E_1,E_2)$ be any normed ideal of operators. 

Given $A\in \mathcal I$, $T\in L(E_0,E_3)$, $K\in K(E_0,E_1)$ a compact operator  and $b \in E_2$,
we may consider the following problem: find the set of solutions of
\begin{equation*} 
\underset{X\in \mathcal I, \ x\in E_0}{\min }f(\|Tx\|_{E_3},\|A-X\|_{\mathcal I},\|XKx-b \|_{E_2})
\end{equation*} 
with $f:\mathbb R_{\ge 0}^3\to\mathbb R_{\ge 0}$ is any continuous, increasing and coercive function.

The fact that $K$ is a fixed compact operator implies that $(Kx_\lambda)_\lambda$ is norm convergent for any weakly convergent net $(x_\lambda)_\lambda$. Thus the bilinear mapping $(X,x)\mapsto XKx$ is weak-to-weak continuous. Therefore, proceeding as in Proposition \ref{restringido a compactos} (2) and (3), it can be shown that the above problem admits solution whenever $\mathcal I$ and $E_0$ are reflexive or $\mathcal I=L(E_1,E_2)$ and $E_0$ and $E_2$ are reflexive.

A typical example where this result can be applied is when $E_0$ is a Sobolev space, $E_1$ is an appropriate $L^p$ space and $K$ is the inclusion $E_0\hookrightarrow E_1$. In this case, the well-known Rellich-Kondrachov theorem assures the compactness of the inclusion.

\subsection{Counterexamples to the weak continuity of the bilinear mapping}

Restricting the \eqref{RTLS} problem to different types of operators we would obtain slightly different bilinear mappings. The following examples show that these bilinear mappings are usually not weak-to-weak continuous.
\begin{example}
    On the Hilbert space $L^2(\Omega)$, consider an integral Hilbert-Schmidt operator,
    $$
    A_0f(s)=\int_{\Omega}k_0(s,t)f(t)dt,
    $$
    where $k_0\in L^2(\Omega^2)$. Recall that $\|A_0\|_2=\|k_0\|_{L^2(\Omega^2)}$. We show that the bilinear mapping
    \begin{align*}
    B:L^2(\Omega^2)\times L^2(\Omega) &\to L^2(\Omega) \\
    (k,f) & \mapsto  B(k,f)(s)=\int_{\Omega}k(s,t)f(t)dt,
\end{align*}
is not weak-to-norm continuous. Let $(e_n)_n$ be any orthonormal basis of $L^2(\Omega)$ (thus a weakly null sequence) and take $g\in L^2(\Omega)$. Define $k_n(s,t)=g(s)e_n(t)$. It is easy to see that $(k_n)_n$ is weakly null in $L^2(\Omega^2)$. 

But 
$$
B(k_n,e_n)(s)=g(s)\int_{\Omega}e_n(t)e_n(t)dt=g(s).
$$
Therefore, $k_n\overset{w}{\to}0$ and $e_n\overset{w}{\to}0$ but $B(k_n,e_n)=g$ for every $n$.
\end{example}

\smallskip
Recently, in \cite{sixou2019kullback} it was proposed to study a variant of the problem by restricting the set of vectors to the weakly compact set $D\subset L^1$, consisting on all the functions whose essential image is contained in $[d_1,d_2]$, with $0<d_1<d_2$. We see in the following example that, in the context of integral operators, the weak-to-weak continuity of the bilinear mapping $L^2\times D\to L^1$ is not satisfied.

\begin{example}
 On $L^p(\Omega)$, $\Omega=[0,2\pi]$, consider the problem \eqref{restricted RTLS} associated to an integral operator,
    $$
    A_0f(s)=\int_{\Omega}k_0(s,t)f(t)dt,
    $$
    where $k_0\in L^q(\Omega^2)$, for $1\le p,q < \infty$. We prove that, for some weakly compact subsets $C\subset L^q(\Omega^2)$ and $D\subset L^p(\Omega)$, the bilinear mapping     \begin{align*}
    B:C\times D &\to L^p(\Omega) \\
    (k,f) & \mapsto  B(k,f)(s)=\int_{\Omega}k(s,t)f(t)dt,
\end{align*}
is not weak-to-weak continuous. 

In fact, let $f_n(t)=2-\cos(nt)\in L^p([0,2\pi])$ and $k_n(s,t)=2+\cos(nt)\in L^q([0,2\pi]^2)$, for arbitrary $p,q$. Note that both $f_n$ and $k_n$ are uniformly bounded above and below by 3 and 1 respectively, thus both sequences are contained in weakly compact sets. Moreover, since  $\cos(nt)$ converge weakly to 0 in  $ L^p([0,2\pi])$ for any $p$, we have that 
$$
f_n\overset{w}{\to} 2, \qquad\textrm{and}\qquad k_n\overset{w}{\to} 2,
$$
(where 2 denotes the constant function). Then 
$$
B(2,2)=B(w -\lim k_n,w-\lim f_n)=\int_0^{2\pi}4=8\pi.
$$
But, on the other hand, 
$$
B(k_n,f_n)(s)=\int_0^{2\pi}(2+\cos(n t))(2-\cos(n t))dt=8\pi-\int_0^{2\pi}\cos^2(n t)dt=7\pi.
$$
Therefore, $B$ is not weak to weak continuous.

 Note also that a similar reasoning may have be done when either $p$ or $q$ equal $\infty$ replacing the weak topology by the weak$^*$ topology.
\end{example}

\begin{remark}
The above example shows, once again, that the existence of solution of \eqref{restricted RTLS} is a subtle problem. 
Indeed, since $L^1(\Omega)$ has the  Dunford-Pettis property, if we take in the above example $C$ to be a weakly compact set of bounded operators on $L^1(\Omega)$ (instead of a weakly compact set of kernels in $L^q(\Omega^2)$) then, by Proposition \ref{prop dunford-pettis},  the bilinear mapping has the necessary continuity in order to prove the existence of solution of \eqref{restricted RTLS}. 
\end{remark}



\bigskip
To assure the existence of solution for problem \eqref{restricted RTLS}, in the case when non-reflexive spaces are involved,  one would have to prove that the associated bilinear mapping is weak$^*$-continuous (or $WOT$-continuous in the case of $L(\HH)$).
Our last example shows that if we do not restrict the domain then the bilinear mapping $B(X,y)=Xy$ lacks the desired continuity  
on a very general situation.
\begin{example}
Let $E_0$ be any infinite dimensional Banach space.
Then, there is a weakly null net $(y_\lambda)_\lambda$ contained in the sphere of $E_0$. Take now $y_\lambda'\in E_0^*$ of norm 1 such that $y_\lambda'(y_\lambda)=1$. By taking a subnet, we may suppose that $y_\lambda'$ converge in the weak$^*$ topology to some $y'$ in the closed unit ball of $E_0^*$. Take $0\ne y_1\in E_1$. Then $(y_\lambda'(\cdot)y_1)_\lambda$ is a  net in $L(E_0,E_1)$ that is $WOT$-convergent to $y'(\cdot)y_1$. Moreover $B(y_\lambda'(\cdot)y_1,y_\lambda)=y_1$ for every $\lambda$ and $B(y'(\cdot)y_1,0)=0$. Therefore $B$ is not $WOT\times $weak  continuous.
\end{example}

\section*{Acknowledgments}
Maximiliano Contino was supported by the UBA's Strategic Research Fund 2018 and CONICET PIP 0168.  Alejandra Maestripieri was supported by CONICET PIP 0168.  
Guillermina Fongi was supported by PICT 2017 0883.	
\goodbreak


\begin{thebibliography}{10}
	
	\bibitem{aiken1980unitary}
	J.~G. Aiken, J.~A. Erdos, and J.~A. Goldstein, \emph{Unitary approximation of
		positive operators}, Illinois J. Math., \textbf{24} (1980), 61--72.
	
	\bibitem{aqzzouz2017dunford}
	B.~Aqzzouz and K.~Bouras, \emph{Dunford-Pettis sets in Banach lattices}, Acta
	Math. Univ. Comenianae, \textbf{81} (2017), 185--196.
	
	\bibitem{atteia1965generalisation}
	M.~Atteia, \emph{G{\'e}n{\'e}ralisation de la d{\'e}finition et des
		propri{\'e}t{\'e}s des spline fonctions}, Compt. Rend. Acad. Sci.,
	\textbf{260} (1965), 3550--3553.
	
	\bibitem{beck2006solution}
	A.~Beck and A.~Ben-Tal, \emph{On the solution of the Tikhonov regularization of
		the total least squares problem}, SIAM J. Optim., \textbf{17} (2006),
	98--118.
	
	\bibitem{bjorck1996numerical}
	A.~Bj{\"o}rck, \emph{Numerical methods for least squares problems},
	Philadelphia, PA: SIAM, 1996.
	
	\bibitem{bleyer2013double}
	I.~R. Bleyer and R.~Ramlau, \emph{A double regularization approach for inverse
		problems with noisy data and inexact operator}, Inverse Problems, \textbf{29}
	(2013), 025004.
	
	\bibitem{bleyer2015alternating}
	I.~R. Bleyer and R.~Ramlau, \emph{An alternating iterative minimisation
		algorithm for the double-regularised total least square functional}, Inverse
	Problems, \textbf{31} (2015), 075004.
	
	\bibitem{castillo1994dunford}
	J.~Castillo and M.~Gonz{\'a}les, \emph{On the Dunford-Pettis property in Banach
		spaces}, Acta Univ. Carolin. Math. Phys., \textbf{35} (1994), 5--12.
	
	\bibitem{[ChaLenMil96]}
	R.~Champion, C.~Lenard, and T.~Mills, \emph{An introduction to abstract
		splines}, Math. Sci., \textbf{21} (1996), 8--26.
	
	\bibitem{[ChaLenMil00]}
	R.~Champion, C.~T. Lenard, and T.~M. Mills, \emph{A variational approach to
		splines}, ANZIAM J., \textbf{42} (2000), 119--135.
	
	\bibitem{contino2019global}
	M.~Contino, M.~E. Di~Iorio~y Lucero, and G.~Fongi, \emph{Global solutions of
		approximation problems in Hilbert spaces}, Linear and Multilinear Algebra,
	(2019), 1--17.
	
	\bibitem{continoPolyak}
	M.~Contino, G.~Fongi, and S.~Muro, \emph{Polyak's theorem on Hilbert spaces},
	In preparation,  (2020).
	
	\bibitem{corach2016optimal}
	G.~Corach, G.~Fongi, and A.~Maestripieri, \emph{Optimal inverses and abstract
		splines}, Linear Algebra Appl., \textbf{496} (2016), 182--192.
	
	\bibitem{corach2002oblique}
	G.~Corach, A.~Maestripieri, and D.~Stojanoff, \emph{Oblique projections and
		abstract splines}, J. Approx. Theory, \textbf{117} (2002), 189--206.
	
	\bibitem{defant1992tensor}
	A.~Defant and K.~Floret, \emph{Tensor norms and operator ideals}, Elsevier,
	1992.
	
	\bibitem{diestel1980survey}
	J.~Diestel, \emph{A survey of results related to the Dunford-Pettis property},
	Contemp. Math., \textbf{2} (1980), 15--60.
	
	\bibitem{dinkelbach1967nonlinear}
	W.~Dinkelbach, \emph{On nonlinear fractional programming}, Manag. Science,
	\textbf{13} (1967), 492--498.
	
	\bibitem{golub1999tikhonov}
	G.~H. Golub, P.~C. Hansen, and D.~P. O'Leary, \emph{Tikhonov regularization and
		total least squares}, SIAM J. Matrix Anal. Appl., \textbf{21} (1999),
	185--194.
	
	\bibitem{golub1979total}
	G.~H. Golub and C.~Van~Loan, \emph{Total least squares}, in \emph{Smoothing
		Techniques for Curve Estimation}, vol. 757, New York, NY: Springer-Verlag,
	1979, 69--76.
	
	\bibitem{hanche2010kolmogorov}
	H.~Hanche-Olsen and H.~Holden, \emph{The Kolmogorov--Riesz compactness
		theorem}, Expo. Math., \textbf{28} (2010), 385--394.
	
	\bibitem{kalton1974spaces}
	N.~J. Kalton, \emph{Spaces of compact operators}, Math. Ann., \textbf{208}
	(1974), 267--278.
	
	\bibitem{nashed1987inner}
	M.~Z. Nashed, \emph{{Inner, outer, and generalized inverses in Banach and
			Hilbert spaces.}}, Numer. Funct. Anal. Optim., \textbf{9} (1987), 261--325.
	
	\bibitem{nguyen2019solving}
	H.~Q. Nguyen, R.~L. Sheu, and Y.~Xia, \emph{Solving a Type of the Tikhonov
		Regularization of the Total Least Squares by a New S-Lemma}, in \emph{World
		Congress on Global Optimization}, Springer, 2019, 221--227.
	
	\bibitem{penot1982semi}
	J.~P. Penot and M.~Th{\'e}ra, \emph{Semi-continuous mappings in general
		topology}, Arch. Math., \textbf{38} (1982), 158--166.
	
	\bibitem{pietsch1978operator}
	A.~Pietsch, \emph{Operator ideals}, vol.~16, Deutscher Verlag der
	Wissenschaften, 1978.
	
	\bibitem{rao1973theory}
	C.~R. Rao and S.~K. Mitra, \emph{Theory and application of constrained inverse
		of matrices}, SIAM J. Appl. Math., \textbf{24} (1973), 473--488.
	
	\bibitem{ringrose1971compact}
	J.~R. Ringrose, \emph{Compact non-self-adjoint operators}, vol.~35, New York,
	NY: Van Nostrand Reinhold Co., 1971.
	
	\bibitem{serrano2006weakly}
	E.~Serrano, C.~Pi{\~n}eiro, and J.~Delgado, \emph{Weakly equicompact sets of
		operators defined on Banach spaces}, Arch. Math., \textbf{86} (2006),
	231--240.
	
	\bibitem{serrano2007some}
	E.~Serrano, C.~Pi{\~n}eiro, and J.~Delgado, \emph{Some properties and
		applications of weakly equicompact sets}, Arch. Math., \textbf{89} (2007),
	266--277.
	
	\bibitem{sixou2019kullback}
	B.~Sixou and C.~Mory, \emph{Kullback-Leibler residual and regularization for
		inverse problems with noisy data and noisy operator.}, Inverse Problems \&
	Imaging, \textbf{13} (2019), 1113–1137.
	
	\bibitem{Tikhonov}
	A.~N. Tikhonov and V.~Y. Arsenin, \emph{Solutions of ill-posed problems},
	Washington, DC: Winston, 1977.
	
	\bibitem{van1991total}
	S.~Van~Huffel and J.~Vandewalle, \emph{The total least squares problem:
		computational aspects and analysis}, vol.~9, Philadelphia, PA: Siam, 1991.
	
	\bibitem{weidmann2012linear}
	J.~Weidmann, \emph{Linear operators in Hilbert spaces}, vol.~68, New York, NY:
	Springer Science \& Business Media, 2012.
	
\end{thebibliography}
%

\end{document}